\documentclass[11pt,pdftex]{article}
\RequirePackage{fix-cm}
\usepackage[T1]{fontenc}
\usepackage{url}
\usepackage{graphics}
\usepackage{graphicx}
\usepackage{diagbox}
\usepackage{slashbox}
\usepackage{color}
\usepackage{amsmath}
\usepackage{amsfonts}
\usepackage{amsthm,amssymb}
\usepackage{tikz,color}
\usepackage{tabularx}
\usepackage{xcolor}
\usepackage{colortbl}

\allowdisplaybreaks

\def\vector#1{{\mathbf {#1}}}

\def\oddeven#1#2{{\text{\rm oe}_{#1,#2}}}
\def\oddodd#1#2{{\text{\rm oo}_{#1,#2}}}
\def\evenodd#1#2{{\text{\rm eo}_{#1,#2}}}
\def\eveneven#1#2{{\text{\rm ee}_{#1,#2}}}

\def\nmodk#1#2{{(#1\,\text{\rm mod}\, #2)}}

\def\M{{\cal M}}

\def\eve{{\,\text{\rm even}}}
\def\odd{{\,\text{\rm odd}}}

\def\CM#1#2{{\ucr(#1,#2)}}

\newcommand{\rr}{{\mathcal R}}

\newcommand{\Cr}{{\hbox{\rm cr}}}
\newcommand{\ucr}{{\hbox{\rm cr}}}
\newcommand{\Crt}{{{\nu}_2}} 

\newcommand\floor[1]{{\lfloor{#1}\rfloor}}
\newcommand\ceil[1]{{\lceil{#1}\rceil}}
\newcommand{\bigfloor}[1]{{\biggl\lfloor{#1}\biggr\rfloor}}

\newcommand{\Diag}{\mbox{Diag}}

\newcommand{\trace}{\mathrm{trace}}

\newcommand{\beq}{\begin{equation}}
\newcommand{\eeq}{\end{equation}}
\newcommand{\beann}{\begin{eqnarray*}}
\newcommand{\eeann}{\end{eqnarray*}}
\newcommand{\bc}{\begin{center}}
\newcommand{\ec}{\end{center}}

\theoremstyle{definition}
\newtheorem{theorem}{Theorem}
\newtheorem{observation}[theorem]{Observation}
\newtheorem{lemma}[theorem]{Lemma}
\newtheorem{claim}[theorem]{Claim}

\newtheorem{proposition}[theorem]{Proposition}

\newtheorem{conjecture}[theorem]{Conjecture}
\title{Improved lower bounds on book crossing numbers of complete graphs.}
\author{E.~de Klerk\thanks{Department of Econometrics and OR, Tilburg University, The Netherlands.}
  \and
D.V.~Pasechnik\thanks{School of Physical and Mathematical Sciences, Nanyang Technological University,
             Singapore. Supported by Singapore Ministry of Education ARF Tier 2 Grant MOE2011-T2-1-090.}
             \and
             G.~Salazar\thanks{Instituto de Fisica,
Universidad Autonoma de San Luis Potosi,
San Luis Potosi, SLP Mexico 78000. Supported by CONACYT Grant 106432.
}}

\begin{document}
\maketitle
\begin{abstract}
A {\em book with $k$ pages} consists of a straight line (the {\em
  spine}) and $k$ half-planes (the {\em pages}), such that the
boundary of each page is the spine. If a graph is drawn on a book
with $k$ pages in such a way that the vertices lie on the spine, and
each edge is contained in a page, the result is a {\em k-page book
  drawing} (or simply a {\em $k$-page drawing}). The $k$-{\em
  page crossing number} $\nu_k(G)$ of a graph $G$ is the minimum
number of crossings in a $k$-page drawing of $G$. In this paper
we investigate the $k$-page crossing numbers of  complete graphs. We use
semidefinite programming techniques to give improved lower bounds on
 $\nu_k(K_n)$ for various values of $k$. We also use a maximum
satisfiability reformulation to 
calculate the exact
value of $\nu_k(K_n)$ for several values of $k$ and $n$.
Finally, we investigate the best construction known for drawing $K_n$
in $k$ pages, calculate the resulting number of crossings, and discuss
this upper bound in the light of the new results reported in this paper.


\end{abstract}


\section{Introduction}\label{sec:introduction}

Motivated by applications to VLSI design, Chung, Leighton and
Rosenberg~\cite{leighton} studied embeddings of graphs in books. A
{\em book} consists of a line (the {\em spine}) and $k\ge 1$
half-planes (the {\em pages}), such that the boundary of each page is
the spine. In a {\em book embedding}, each edge is drawn on a single
page, and no edge crossings are allowed. In a {\em book drawing} (or
{\em $k$-page drawing}, if the book has $k$ pages), each
edge is drawn on a single page, but edge crossings are allowed.

Obviously every fixed graph can be embedded in a book with sufficiently many
pages. On the other hand, it is desirable to do so using as few pages
as possible. Given a graph $G$, the minimum $k$ such that $G$ can be
embedded in a $k$-page book is the {\em pagenumber} (or {\em book
  thickness}) of $G$~\cite{Bernhart-Kainen,leighton,Kainen}. Determining the pagenumber of an arbitrary graph is
NP-Complete~\cite{leighton}, but some results are known for particular
families of graphs. For instance, it is not
difficult to prove that the pagenumber of the complete graph $K_n$ is
$\lceil{n/2}\rceil$. On the other hand, with few exceptions, the
  pagenumbers of the complete bipartite graphs $K_{m,n}$ are unknown
  (see~\cite{Enomoto,muder}). Yannanakis proved~\cite{Yannanakis} that
  four pages are always sufficient, and sometimes required, to embed a
  planar graph.

\subsection{The $k$-page crossing number $\nu_k(G)$ of a graph $G$}

When the number $k$ of pages is fixed, the goal is to minimize the number
of crossings in a $k$-page drawing of an input graph. The {\em
  $k$-page crossing number} $\nu_k(G)$ of a graph $G$ is the minimum
number of crossings in a $k$-page drawing of $G$.

Clearly, a graph $G$ has $\nu_1(G) = 0$ if and only if it is
outerplanar. Equivalent to $1$-page drawings are {\em circular
  drawings}, in which the vertices are placed on a circle and
all edges are drawn in its interior.
In a similar vein, $k$-page drawings of $G = (V,E)$ can be
alternatively viewed as a set of $k$ circular drawings of graphs
$G^{(i)} = (V, E^{(i)})$ $(i=1,\ldots,k)$, where the edge sets
$E^{(i)}$ form a $k$-partition of $E$.  In other words, we assign each
edge in $E$ to exactly one of the $k$ circular drawings. In
Figure~\ref{fig:K7draw} we illustrate a $3$-page drawing of $K_7$.

\begin{figure}[ht!]
\begin{center}
\includegraphics[width=14cm]{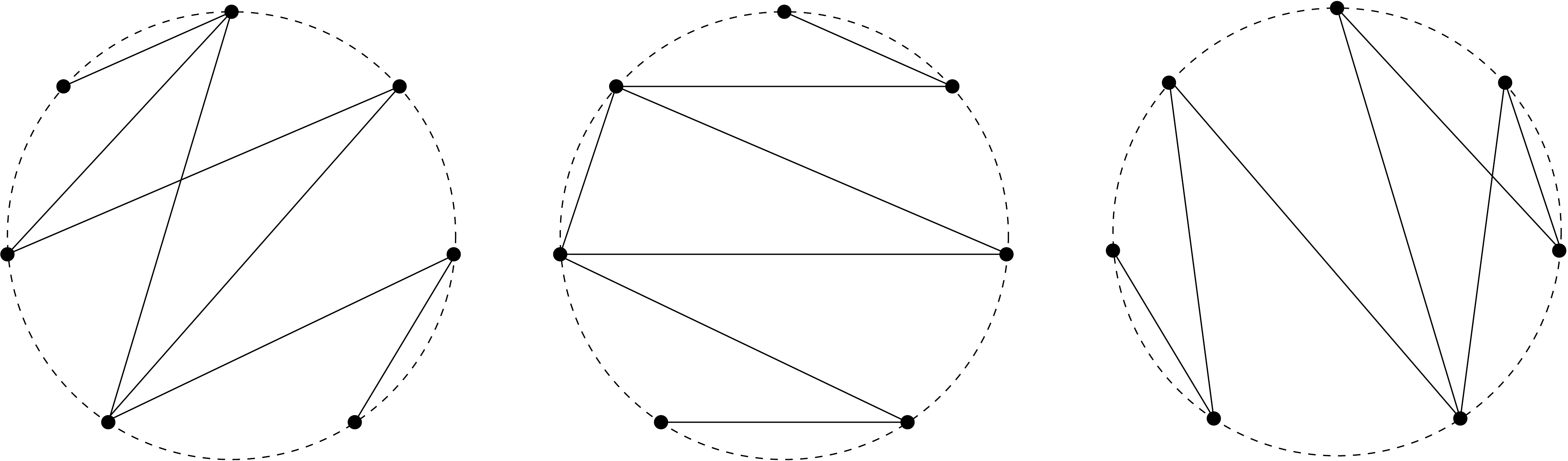}
\caption{\label{fig:K7draw} A $3$-page drawing of $K_7$ with $2$ crossings.}
\end{center}
\end{figure}

Several computational approaches and heuristics for estimating (or
determining) $\nu_1(G)$ and $\nu_2(G)$ have been devised (see for
instance~\cite{he1,he2,he5,masuda0,masuda}). Few exact results or
nontrivial bounds are known, and very little is known about $\nu_k(G)$
for $k > 2$.



Although the special cases $k=1$ and $2$ have received considerable
attention, the only thorough investigation of $\nu_k(G)$ for arbitrary
$k$ is the work by Shahrokhi, S\'ykora, Sz\'ekely, and
Vrt'o~\cite{sssv96}. In this paper, Shahrokhi et al.~give general
lower bounds for $\nu_k(G)$, for any graph $G$. They also give lower and upper bounds for
$\nu_k(K_n)$, and use their upper bounds for $\nu_k(K_n)$ to give
general upper bounds for $\nu_k(G)$ for arbitrary graphs $G$. 

As with every graph-theoretical parameter, there is a natural interest
in computing (or at least estimating) the $k$-page crossing number of
the complete graph $K_n$. Besides, estimates on $\nu_k(K_n)$ are an
essential tool to derive bounds for $\nu_k(G)$ for other graphs $G$,
via the embedding method. This is the approach followed by Sharokhi et
al.~in~\cite{sssv96}, where the constructions that yield their upper bounds for
$\nu_k(K_n)$ are used to generate $k$-page drawings of dense
graphs, whose number of crossings is within a constant factor of their
$k$-page crossing number.



\subsection{Structure of the rest of the paper}

Our main contributions in this paper are improved lower bounds for the
$k$-page crossing numbers of $K_n$. We also compute the exact value
of $\nu_k(K_n)$ for several $k$ and $n$ (no exact
values were previously known for any $n$, for any $k > 2$).

In Section~\ref{sec:surv1} we survey the bounds and exact results
known for $\nu_k(K_n)$.  In Section \ref{sec:maxcutformulation} we
show how $\nu_k(K_n)$ may be obtained from the solution of a
maximum-$k$-cut problem on a suitable graph, or via the solution of a
suitable weighted maximum satisfiability problem.  These
reformulations are used in Section \ref{sec:numerical results} to
obtain previously unknown exact values and improved lower bounds on $\nu_k(K_n)$ for various
values of $k$ and $n$ via computation.
In Section~\ref{sec:upperbounds} we review the construction that gives
the best upper bounds available for $\nu_k(K_n)$, calculate the resulting number of
crossings, and analyze this upper bound in the light of the new
results obtained in the previous sections. Finally, in
Section~\ref{sec:concludingremarks} we present some concluding remarks
and open questions.

\section{$k$-page drawings of $K_n$: exact results and bounds}\label{sec:surv1}

\subsection{$1$-page drawings of $K_n$}\label{sec:1pkn}

Calculating the $1$-page crossing number of $K_n$ is straightforward.
 Indeed, it is obvious that, in every $1$-page drawing of
$K_n$, every four vertices define a crossing, and therefore $\nu_1(K_n) \ge
\binom{n}{4}$. It is easy to give $1$-page drawings with exactly
$\binom{n}{4}$ crossings, and so the reverse inequality $\nu_1(K_n)
\le \binom{n}{4}$ follows.

It follows that the problem of calculating or estimating $\nu_k(K_n)$
is only of interest for $k\ge 2$.

\subsection{$2$-page drawings of $K_n$}

We recall that the {\em crossing number} $\Cr(G)$ of a graph $G$ is
the minimum number of crossings in a drawing of $G$ in the plane.
Harary and Hill~\cite{hararyhill} described how to draw $K_n$ in the
plane with $Z_2(n)$ crossings, where
\begin{equation}\label{eq:hill}
Z_2(n):= \frac{1}{4}\biggl\lfloor{\frac{n}{2}}
\biggr\rfloor
\biggl\lfloor{\frac{n-1}{2}}\biggr\rfloor
\biggl\lfloor{\frac{n-2}{2}}\biggr\rfloor
\biggl\lfloor{\frac{n-3}{2}}\biggl\rfloor.
\end{equation}

No drawings of $K_n$ with fewer than $Z_2(n)$ crossings are known, and
to this date the Harary-Hill Conjecture $\Cr(K_n)=Z_2(n)$ is still open
(it has been settled only for $n\le 12$; see~\cite{panrichter}).

The drawings given in~\cite{hararyhill} are not $2$-page drawings, but
it was later noticed that $2$-page drawings with $Z_2(n)$ crossings do
exist~\cite{bk} (see also~\cite{guy2,Harborth}). This observation gave rise
to the conjecture $\nu_2(K_n) = Z_2(n)$, popularized by Vrt'o~\cite{vrto2}.

Buchheim and Zheng~\cite{Buchheim-Zheng} proved that
$\nu_2(K_n)=Z_2(n)$ for $n \le 14$.
Recently, De Klerk and
Pasechnik~\cite{dp} verified that $\nu_2(K_n)=Z_2(n)$ for
$n\le 18$ and $n=20$ and $24$, and used semidefinite programming
techniques to give asymptotic estimates on $\nu_2(K_n)/Z_2(n)$.
More recently, \'Abrego et al.~proved that $\nu_2(K_n) = Z_2(n)$ for
every $n$~\cite{abregoetal}.

%
%
%
%

\subsection{$k$-page drawings of $K_n$ for $k\ge 3$: upper bounds}\label{sub:k3up}

Much less is known of the $k$-page crossing number $\nu_k(K_n)$ for $k>2$.

As we mentioned above, Bla\v{z}ek and Koman~\cite{bk} seemed to have
been the first to
construct $2$-page drawings of $K_n$ with $Z_2(n)$ crossings. In the
same paper they briefly observed that their 
construction could be extended to
$k$ pages. Although no details were given, they gave the 
number of crossings obtained for the case $k=3$ (see Section~\ref{sub:keq3}).

Damiani, D'Antona and Salemi proposed a way to draw $K_n$ on $k$
pages~\cite{DaDASa}, using the adjacency matrix representation (we
call this the {\em DDS construction}). 
They included a table
with the resulting number of crossings for $n\le 18$, and all $k\le
\ceil{n/2}$ (recall that $\nu_k(K_n)= 0$ if $k > \ceil{n/2}$).  The
exact number of crossings resulting from their construction was not
explicitly given (we have calculated this number; see
Proposition~\ref{pro:calcu}).
The
(table) results given in~\cite{DaDASa} coincide for the case $k=3$ with the expression given by
Bla\v{z}ek and Koman. Although Bla\v{z}ek and Koman did not explain in
detail their proposed construction for $k > 2$, it is not difficult to
fill out the details; by doing so, one can confirm that the (or, at
least, one possible) natural
way to generalize their construction for $k > 2$ yields precisely the DDS construction.

In~\cite{sssv96}, Shahrokhi et al.~described a construction that draws
$K_n$ on $k$ pages. This is also a natural generalization of the Bla\v{z}ek-Koman 
construction to $k>2$ pages. Moreover, it agrees with the DDS construction whenever $k$
divides $n$ (the DDS construction yields slightly better
results for other values of $k$ and $n$).
Based on their construction, Shahrokhi et al.~gave the following
general upper bound:
\begin{equation}
\label{eq:nuk upper bound}
\nu_k(K_n) \le
\frac{2}{k^2}\left(1-\frac{1}{2k}\right)\binom{n}{4} +
\frac{n^3}{2k}.
\end{equation}

In Section~\ref{sec:upperbounds} we include a detailed discussion on
the DDS construction, including the
calculation of the number of crossings that result by drawing $K_n$ on
$k$ pages using this paradigm.

\subsection{$k$-page drawings of $K_n$ for $k \ge 3$: lower bounds}

Shahrokhi et al.~proved in~\cite{sssv96} that for every graph $G$
and every positive integer $k$, one has $\nu_k(G) \ge m^3/37k^2n^2 -
27kn/37$. Following the derivation of this bound, it is easy to see
that the factor $1/37$ in this expression can be improved, but only
marginally so. Applying this bound to $K_n$, we obtain
$$
\nu_k(G) \ge \frac{n(n-1)^3}{296k^2} - \frac{27kn}{37}=
\frac{3}{37k^2}\binom{n}{4} + O(n^3).
$$

This lower bound can be improved if $n$ is sufficiently large compared
to $k$, as follows.
We recall that a $k$-{planar drawing} is similar to a $k$-page
drawing, but involves $k$ unrestricted planar drawings. Formally, let
$G=(V,E)$ be a graph. A $k$-{\em planar drawing} of $G$ is a set of $k$
planar drawings of graphs $G^{(i)} = (V, E^{(i)})$ $(i=1,\ldots,k)$, where the edge sets
$E^{(i)}$ form a $k$-partition of $E$. Loosely speaking, to obtain the
$k$-planar drawing, we take the
drawings of the graphs $G^{(i)}$, and (topologically) identify the $k$
copies of each vertex. The $k$-{\em planar crossing number}
$\Cr_k(G)$ of $G$ is the minimum number of crossings in a $k$-planar
drawing of $G$.

If $k$ is even, then it is easy to obtain, from a $k$-page drawing of
a graph $G$, a ${k/2}$-planar
drawing of $G$ with the same number of crossings. Therefore $\nu_k(G) \ge \Cr_{{k/2}}(G)$
for every graph $G$ and any positive even integer $k$. In~\cite{sssv07}, Shahrokhi et
al.~proved that for all $n \ge 2r^2 + 6r -1$ and all $r \ge 1$,
\begin{equation*}
\label{eq:crk lower bound}
\Cr_r(K_n) \ge \frac{1}{2(3r-1)^2}\binom{n}{4}.
\end{equation*}

From our previous observation, it follows that for all $n \ge 2(k/2)^2 +
6(k/2) -1 = k^2/2 + 3k - 1$ and all even $k \ge 2$, $\nu_k(K_n) \ge \frac{2}{(3k-2)^2}\binom{n}{4}$.
Obviously $\nu_{k-1}(G)\ge
\nu_{k}(G)$ for any graph $G$ and any integer $k\ge 2$, and so for any
odd $k \ge 3$ and any $n \ge (k-1)^2/2 + 3(k-1) - 1 = k^2/2 + 2k - 7/2$ we
have
$\nu_k(K_n) \ge \nu_{k+1}(K_n) \ge
\frac{2}{(3(k+1)-2)^2}\binom{n}{4} =
 \frac{2}{(3k+1)^2}\binom{n}{4}$.
Now in their exhaustive investigation of biplanar ($2$-planar) crossing
numbers~\cite{bip1,bip2}, Czabarka, S\'ykora, Sz\'ekely, and Vrt'o prove the slightly
better bound (for the $2$-planar, or {\em biplanar} crossing number) $\Cr_2(K_n) \ge n^4/952$. From this it follows that
$\nu_4(K_n) \ge n^4/952$.
Putting all these results
together, we obtain the following lower bounds:

\begin{equation}
\label{eq:nuk lower bound}
\nu_k(K_n) \ge
\begin{cases}
\frac{3}{119}\binom{n}{4} + O(n^3), & \text{if $k=4$;}
\\
\frac{2}{(3k-2)^2}\binom{n}{4}, & \text{if $k$ is even, $k > 4$, and
$n \ge k^2/2 + 3k -1$
;}
\\
\frac{2}{(3k+1)^2}\binom{n}{4}, & \text{if $k$ is odd, and $n \ge k^2+2k-7/2$.}
\end{cases}
\end{equation}

\subsection{$k$-page drawings of $K_n$: asymptotic lower and upper bounds}\label{sub:lus}

The following type of result is well-known and easily shown; see
e.g. \cite{Richter-Thomassen} or \cite[Theorem 2]{Scheinerman-Wilf}.
\begin{claim}\label{cla:claimA}
For any  integers $k >0$ and  $n > m \ge 4$,
\[
\frac{\nu_k(K_n)}{ \binom{n}{4}} \ge  \frac{\nu_k(K_m)}{\binom{m}{4}}.
\]
\end{claim}
As a consequence, the sequence $\frac{\nu_k(K_n)}{ \binom{n}{4}}$ is
monotonically non-decreasing in $n$.  Since it is also bounded from above by
\eqref{eq:nuk upper bound}, the limit exists, and using \eqref{eq:nuk
  upper bound} and \eqref{eq:nuk lower bound} one has:
\begin{equation}\label{eq:forfour}
\frac{3}{119} \le \lim_{n \rightarrow \infty} \frac{\nu_4(K_n)}{
  \binom{n}{4}} \le \frac{7}{64};
\end{equation}

\begin{equation}
\label{eq:foreven}
\frac{2}{(3k-2)^2} \le \lim_{n \rightarrow \infty} \frac{\nu_k(K_n)}{
  \binom{n}{4}} \le \frac{2}{k^2}\left(1-\frac{1}{2k}\right), \text{
  if $k$ is even, } k > 4.
\end{equation}

\begin{equation}\label{eq:forodd}
\frac{2}{(3k+1)^2} \le \lim_{n \rightarrow \infty} \frac{\nu_k(K_n)}{
  \binom{n}{4}} \le \frac{2}{k^2}\left(1-\frac{1}{2k}\right), \text{
  if $k$ is odd, } k \ge 3.
\end{equation}

\section{Formulating $\nu_k(K_{n})$ as a maximum $k$-cut \\
  or maximum satisfiability problem}\label{sec:maxcutformulation}

We will show that $\nu_k(K_{n})$ can be obtained by computing the maximum $k$-cut size
in a certain graph $G_n=(V_n,E_n)$, say, which is a certain subgraph of the complement of the line graph of $K_n$.
The same graph was used in \cite{dp} to investigate $\nu_2(K_n)$, and the general construction of graphs of this type
is due to Buchheim and Zheng \cite{Buchheim-Zheng}.

To define the graph $G_n=(V_n,E_n)$, we consider a Hamiltonian cycle $C_n$ with vertices $v_1, v_2, \ldots, v_n$. Let $V_n$
be the set of {\em chords} of the cycle, that is, the edges $v_iv_j$ with
$v_i$ and $v_j$ at cyclic distance at least $2$. Let us say that
the chords $v_iv_j$ and $v_kv_\ell$ {\em overlap} if $i,k,j,\ell$
occur in this cyclic order as we traverse $C_n$, either in its natural or in
its reverse direction.
Finally, to define $E_n$,
we let two chords $v_iv_j$ and $v_kv_\ell$ be adjacent if they
overlap.

%

Thus one has $|V_n| = \binom{n}{2} - n$, and it is easy to verify
that $|E_n| = \binom{n}{4}$.  The automorphism group of $G_n$ is
isomorphic to the dihedral group $D_n$ on $n$ elements, and there are
$d-1$ orbits of vertices, where $d = \floor{n/2}$. The equivalency
classes of vertices (i.e.\ orbits) may be described as follows: since
vertices correspond to chords in $C_n$, the chords that connect
vertices of $C_n$ at the same cyclic distance belong to the same
equivalency class.  If $n$ is odd, then
the vertices corresponding to chords with cyclic
distance $i$ have valency
$i(i-1) + 2(i-1)(d-i)$, as is easy to check.

For later use, we will label the vertices of $G_n$ so that its adjacency matrix is partitioned into symmetric circulant blocks.
To this end, consider the cycle $C_n$ with vertices numbered $\{0,1,\ldots,n-1\}$ in the usual way.
The vertices of $G_n$ that correspond to chords connecting points at cyclic distance $i$ are now
labeled successively, starting with the chord
\[
\left\{\lfloor n/2 \rfloor \times i \mod n,\left(\lfloor n/2 \rfloor +1\right)\times i \mod n\right\},
\]
and then obtaining the next chords in the ordering via  clockwise cyclic shifts.

Thus the adjacency matrix of $G_n$ is partitioned into a block structure, where
 each row of blocks is indexed by a cyclic distance $i \in \{2,\ldots,d\}$,
 and each block has size $n\times n$.

Moreover, one may readily verify that block $(i,j)$ ($i,j \in \{2,\ldots,d\}, \; i \le j$) is given by the {symmetric} $n \times n$ circulant matrix with first row

\begin{equation}
\label{eq:first row}
[0 \; \mathbf{0}_{\ell_{ij}}^T \; \mathbf{1}_{i-1}^T \; \mathbf{0}_{n-2(i-1)-1-2\ell_{ij}}^T \; \mathbf{1}_{i-1}^T \; \mathbf{0}_{\ell_{ij}}^T],
\end{equation}
where $\mathbf{1}_k$ and $\mathbf{0}_k$ denote the all-ones
and all-zeroes vectors in $\mathbb{R}^k$, respectively, and
\begin{equation}
\label{def:ell_ij}
\ell_{ij} = \left\{ \begin{array}{rl}
d(i-j) \mod n & \mbox{\rm if $i$ and $j$ have the same parity} \\
d(i-j)-j \mod n & \mbox{\rm otherwise.}
\end{array} \right.
\end{equation}
We may now relate the maximum $k$-cut problem for $G_n$ to $\nu_k(K_n)$.

\begin{lemma}\label{lem:reform}
One has
\[
\nu_k(K_{n}) = |E_n| - \mbox{\rm max-$k$-cut}(G_n),
\]
where $\mbox{\rm max-$k$-cut}(G_n)$ denotes the cardinality of a maximum $k$-cut in $G_n$.
\end{lemma}
\proof
First of all, recall that the maximum $k$-cut problem for $G = (V,E)$ may be seen as a vertex coloring problem,
where the vertices $V$ are colored with $k$ colors in such a way that the number of edges with differently colored endpoints
is maximized.
Consider a fixed $k$-page drawing of $K_n$, viewed as $k$ circular drawings. Fix a set of $k$ colors.
Assign the edges on page $i$ of the drawing the $i$th color $(1 \le i \le k)$.
This defines a $k$-partition (or $k$-coloring) of the vertices $V_n$ of $G_n$. Moreover, the number of edges in $E_n$ with endpoints of
the same color equals the number of crossings in the drawing, by construction. \qed

As a consequence of this lemma, one may calculate $\nu_k(K_{n})$ for fixed
(in practice, sufficiently small) values of $n$ by solving a maximum cut problem.  This was done
by Buchheim and Zheng \cite{Buchheim-Zheng} for $k=2$ and  $n \le 13$, by solving
the maximum cut problem with a branch-and-bound algorithm.  Using the
{\tt BiqMac} solver~\cite{BiqMac}, De Klerk and Pasechnik \cite{dp} computed the exact value of $\Crt(K_n)$ for
$n\le 18$ and for $n \in \{20,24\}$.

\subsection{The Frieze-Jerrum max-$k$-cut bound}
\label{sec:GW}

We follow the standard practice to use
$\mathbb{R}^{p\times q}$
(respectively, $\mathbb{C}^{p\times q}$)
to denote the space of $p\times q$ matrices over $\mathbb{R}$
(respectively, $\mathbb{C}$).
For $\vector{A} \in \mathbb{R}^{p \times p}$, the notation
$\vector{A} \succeq 0$ means that $\vector{A}$ is symmetric positive semidefinite,
whereas for
$\vector{A} \in \mathbb{C}^{p\times p}$, it means that $\vector{A}$ is Hermitian positive semidefinite.

Let $G$ be a graph with $p$ vertices,
and let $\vector{L}$ be
its Laplacian matrix.
Frieze and Jerrum introduced the following semidefinite programming-based upper bound
on $\mbox{max-$k$-cut}(G)$:
\begin{equation}\label{eq:gw}
\mathcal{FJ}_k(G) :=  \max \left\{
\frac{k-1}{k}\trace(\vector{L}\vector{X}) \; \biggl| \; \vector{X}
\succeq 0,\ X_{ii} = 1 \; (1 \le i \le p), \; \vector{X} \ge \frac{-1}{k-1}\vector{J} \right\},
\end{equation}
where $\vector{J}$ denotes the all-ones matrix of order $p$.

For $k=2$ this bound coincides with the maximum-cut bound of Goemans and Williamson \cite{goe95}.

The associated dual semidefinite program takes the form:
\begin{equation}
\mathcal{FJ}_k(G) = \min_{\vector{w} \in \mathbb{R}^{p}, \vector{S} \ge 0} \left\{ \sum_{i=1}^p w_i + \frac{1}{k-1}\trace(\vector{J}\vector{S})\; \biggl| \; \Diag(\vector{w}) - \frac{k-1}{2k}\vector{L}-\vector{S} \succeq 0\right\},
\label{dual GW}
\end{equation}
where ${\rm Diag}$ is the operator that maps a $p$-vector to a $p\times p$ diagonal
matrix in the obvious way.

\subsection{The Frieze-Jerrum bound for $G_n$}
Using the technique of symmetry reduction for semidefinite programming (see e.g.\ \cite{GaPa}), one can simplify the
dual  problem (\ref{dual GW}) for the graphs $G_n$ defined in Section~\ref{sec:maxcutformulation}, by using the dihedral automorphism group of $G_n$.
We state the final expression as the following lemma.
The proof is very similar to that of \cite[Lemma 4]{dp}, and we therefore only give an outline.

\begin{lemma}
\label{lemma:reformulation SDP}
Let $n>0$ be an odd integer and $d = \lfloor n/2 \rfloor$. One has
\[
\mathcal{FJ}_k(G_n) = \min_{{y} \in \mathbb{R}^{d-1}} n\sum_{i=2}^d y_i + \frac{n}{k-1}\trace (JX^{(0)})
\]
subject to
\begin{equation}
\label{eq:lhs final lmi}
 \mbox{\rm Diag}\left({y} -  \frac{k-1}{2k} {val}\right) + \Lambda^{(m)} \succeq 0 \; (0\le m \le d),
\end{equation}
where
\begin{eqnarray*}
val_i & = & i(i-1) + 2(i-1)(d-i), \quad 2\le i \le d,  \nonumber \\
\Lambda^{(m)}_{ij}  &=&   \frac{k-1}{k}\sum_{t=\ell_{ij}+1}^{\ell_{ij}+i} e^{\frac{-2\pi mt\sqrt{-1}}{n}}  -X^{(0)}_{ij} - 2\sum_{t=1}^d X^{(t)}_{ij}e^{\frac{-2\pi mt\sqrt{-1}}{n}} \;\;\; {2 \le i\le j \le d}, \label{GWconstraints} \\
\ell_{ij} &=& \left\{ \begin{array}{rl}
d(i-j) \mod n & \mbox{\rm if $i$ and $j$ have the same parity} \\
d(i-j)-j \mod n & \mbox{\rm otherwise}
\end{array} \right. \\
X^{(m)} &=& (X^{(m)})^T \ge  0, \;\;\;  \text{\rm for all } 0 \le m \le d. \nonumber \\
\end{eqnarray*}
\end{lemma}

\proof

Assume that $\vector{w},\vector{S}$ are optimal in \eqref{dual GW} for $G = G_n$ and denote the Laplacian matrix 
of $G_n$ by $\vector{L}$.
We now project
the posititive semidefinite matrix
\[
\Diag(\vector{w}) - \frac{k-1}{2k}\vector{L}-\vector{S} \succeq 0
\]
 onto the centralizer ring of $\mbox{Aut}(G_n)$, via the Reynolds projection operator (or group average), say $\mathcal{R}_{G_n}$:
 \[
 \mathcal{R}_{G_n}(X) := \frac{1}{|\mbox{Aut}(G_n)|} \sum_{P \in \mbox{Aut}(G_n)} P^T X P \quad \quad (X \in \mathbb{R}^{|V_n| \times |V_n|}),
 \]
 where the matrices $P$ are given by the permutation matrix representation of $\mbox{Aut}(G_n)$.
Note that this projection preserves positive semidefiniteness as well as entrywise nonnegativity.
Moreover, as explained in Section \ref{sec:maxcutformulation}, we may assume that 
 $ \mathcal{R}_{G_n}(X)$ is a block matrix consisting of symmetric circulant blocks of order $n$.

Also note that the projection $\mathcal{R}_{G_n}(\Diag(\vector{w}))$ simply averages the components of $\vector{w}$ 
 over the $d-1$ orbits of  $\mbox{Aut$(G_n)$}$.
Denoting the average of the $\vector{w}$ components in orbit $i$ by $y_i$ ($2 \le i \le d$),
 and  $\vector{Z} = \mathcal{R}_{G_n}(\vector{S}) \ge 0$,
we obtain
\begin{equation}
\label{lmi}
\Diag (\vector{y} \otimes \vector{1}_{n}) - \frac{k-1}{2k}\vector{L} - \vector{Z} \succeq 0,
\end{equation}
since $\mathcal{R}_{G_n}(\vector{L}) = \vector{L}$.

Thus we have obtained the reformulation
\[
\mathcal{FJ}_k(G_n) = \min_{{y} \in \mathbb{R}^{d-1},\vector{0} \le  \vector{Z} \in {\mathcal{A}}} \left\{ n\sum_{i=2}^d y_i + \frac{1}{k-1}\trace (\vector{J}\vector{Z}) \; | \; \mbox{s.t. } \eqref{lmi}\right\},
\]
where $\mathcal{A} \subset \mathbb{R}^{|V_n| \times |V_n|}$ denotes the centralizer ring of $\mbox{Aut}(G_n)$, i.e.\ the matrix $*$-algebra
consisting of matrices of order $|V_n|$ that are partitioned into symmetric circulant blocks of order $n$. 

We may now reduce this formulation further by using the discrete Fourier transform matrix to
simultaneously diagonalize the circulant blocks of $\vector{Z}$ and $\vector{L}$.

To this end, let $\vector{Q}$ denote the (unitary) discrete Fourier transform matrix of order $n$.
Condition (\ref{lmi}) is equivalent to
\begin{equation}
\label{lmi2}
(\vector{I}_{d-1} \otimes \vector{Q}) \left(\Diag (\vector{y} \otimes \vector{1}_{n})- \frac{k-1}{2k}\vector{L}- \vector{Z}\right)(\vector{I}_{d-1} \otimes \vector{Q})^* \succeq 0.
\end{equation}
Since the unitary transform involving $\vector{Q}$ diagonalizes any circulant matrix (see e.g.\ \cite{circulant matrices}),
the matrix
$(\vector{I}_{d-1} \otimes \vector{Q})\vector{L}(\vector{I}_{d-1} \otimes \vector{Q})^*$ becomes a block matrix where each $n\times n$ block is diagonal,
with diagonal entries of block $(i,j)$ given
by the eigenvalues of the circulant matrix with first row given by

\[
\left\{
\begin{array}{rll}
& [0 \; \mathbf{0}_{\ell_{ij}}^T \; -\mathbf{1}_{i-1}^T \; \mathbf{0}_{n-2(i-1)-1-2\ell_{ij}}^T \; -\mathbf{1}_{i-1}^T \; \mathbf{0}_{\ell_{ij}}^T] &\mbox{if $i \neq j$} \\
&   [val_i  \; -\mathbf{1}_{i-1}^T \; \mathbf{0}_{n-2(i-1)-1}^T \; -\mathbf{1}_{i-1}^T] & \mbox{if $i = j$}, \\
\end{array}
\right.
\]

due to
 \eqref{eq:first row}. Also, clearly
one has
\[
(\vector{I}_{d-1} \otimes \vector{Q}) \left(\Diag (\vector{y} \otimes \vector{1}_{n}) \right)(\vector{I}_{d-1} \otimes \vector{Q})^* = \Diag (\vector{y} \otimes \vector{1}_{n}).
\]

Finally, the rows and columns of the left hand side of (\ref{lmi2}) may now be re-ordered to form a block diagonal
matrix with $n$ diagonal blocks, each of size $d-1 \times d-1$. Only $d+1$ of these $n$ blocks are distinct, and these correspond to the
left-hand-side matrices in \eqref{eq:lhs final lmi}. The matrices $\Lambda^{(i)}$ $(0\le i \le d)$ in \eqref{eq:lhs final lmi}
correspond to the distinct blocks obtained from
the reordering of
\[
-(\vector{I}_{d-1} \otimes \vector{Q}) \vector{Z}(\vector{I}_{d-1} \otimes \vector{Q})^*
\]
into block-diagonal form. In particular, we use $X^{(m)}_{ij}$ to denote element $m$ of the first row of the symmetric circulant
block $(i,j)$ of $\vector{Z}$.
\qed

A few remarks on Lemma \ref{lemma:reformulation SDP}:
\begin{enumerate}
\item
Note that we obtain a reduced semidefinite program  with $d = \lfloor n/2 \rfloor$ linear matrix inequalities
involving matrices of order $d-1$, as well as $d+1$ nonnegative matrix variables of order $d-1$.
This should be compared to the original formulation \eqref{dual GW} to obtain $\mathcal{FJ}_k(G_n)$,
 that involved a linear matrix inequality of order $n(d-1)$, as well
as a nonnegative matrix variable of the same order. 
\item
Lemma \ref{lemma:reformulation SDP} generalizes \cite[Lemma 4]{dp} to include the case $k > 2$, but also refines it in the sense
that the dihedral symmetry of the graph $G_n$ is fully exploited. Indeed in \cite[Lemma 4]{dp}, only the cyclic
part of $\mbox{Aut}(G_n)$ was used, leading to (complex) Hermitian linear matrix inequalities, as opposed to the
real symmetric linear matrix inequalities of Lemma \ref{lemma:reformulation SDP}.
\item
The computation of $\mathcal{FJ}_k(G_n)$ is simpler in the case $k=2$, since the bound then becomes the Goemans-Williamson maximum cut bound.
Indeed, in \cite{dp}, values of $\mathcal{FJ}_2(G_n)$ were reported for $n$ close to $1,000$.
For $k>2$, one is limited to more modest values: the largest value of $n$ for which we will report computational results will be $n=69$; see Section \ref{sec:numerical results}.
The difference in size of $n$ that may be handled is primarily due to the nonnegative matrix variables $X^{(m)}$ ($0 \le m \le d$). These variables may be eliminated
if $k=2$, but not if $k>2$.
\end{enumerate}

\subsection{A maximum satisfiability reformulation}
It is well-known that the maximum $k$-cut problem may be reformulated as a maximum satisfiability problem,
and we will use this reformulation later on for computational purposes.

Consider a graph $G=(V,E)$ and a set of $k$ colors (used to color the vertices $V$). We define the following logical variables:
\[
x_i^{j} = \left\{\begin{array}{ll}  \mbox{{ \textsc{true}}} & \mbox{ if vertex $i$ has color $j$}
\\
\mbox{{\textsc{false}}} & \mbox{ otherwise}. \end{array} \right.
\]

Consider the clause:
\begin{equation} \label{2-clauses}
\neg x_i^{p} \vee \neg x_j^{p} \mbox{ if $(i,j) \in E$}
\end{equation}
for
each color $p= 1,\ldots,k$.
For a given edge, and a given color, this clause is satisfied if and only if
the endpoints of the edge are not both colored using this color.

Moreover, each vertex should be
assigned a color:
\begin{equation}
  x_i^{1} \vee \ldots \vee x_i^{k} \;\;\; (i \in V). \label{long clauses}
\end{equation}

In order to obtain the maximum $k$-cut in $G$, we therefore need values of the logical variables that
satisfy all the clauses \eqref{long clauses}, and as many of the clauses \eqref{2-clauses} as possible.
This may be done by solving a weighted maximum satisfiability problem, where the weights of the
satisfied clauses is maximized.
In order to guarantee that the clauses \eqref{long clauses} are all satisfied, we assign these clauses weight $k|E|$, while
the clauses \eqref{2-clauses} are assigned weight $1$.

Thus the cardinality of a maximum $k$-cut in $G = (V,E)$ coincides with the maximum weight of
satisfied clauses in a truth assignment for the weighted
logical formula:
\begin{eqnarray}
&&\neg x_i^{p} \vee \neg x_j^{p} \quad \quad ((i,j) \in E, \; 1 \le p \le k)  \nonumber \\
&& k|E|\left(x_i^{1} \vee \ldots \vee x_i^{k}\right) \quad \quad (i \in V). \label{clauses}
\end{eqnarray}

We may now apply this idea to obtain $\nu_k(K_n)$ as follows.
\begin{lemma}
\label{lemma:maxsat}
Consider  the set of  weighted clauses \eqref{clauses} for the graph $G = G_n = (V_n,E_n)$.
Then $\nu_k(K_n)$ is the minimum weight of the unsatisfied clauses, taken over all possible truth assignments.
\end{lemma}
\proof
The proof follows directly from Lemma \ref{lem:reform}. \qed

\section{Numerical results}
\label{sec:numerical results}

\subsection{Exact computations}\label{sub:exact}
It is possible to compute $\nu_k(K_n)$ exactly using software for the weighted maximum satisfiablity problem in Lemma \ref{lemma:maxsat}.
In Table \ref{tab:max sat results} we show results obtained using the solver {\tt Akmaxsat} by K\"ugel \cite{akmaxsat}.

\begin{table}[ht!]
\begin{center}
{\footnotesize
\begin{tabular}{|c||c|c|c|c|c|c|c|c|c|} \hline
\diagbox{$k$}{$n$}      & 7    & 8   & 9   & 10   & 11   & 12   & 13  &14  &15\\ \hline\hline
3                         & 2    & 5   & 9   & 20   & 34   & 51   & 83   & 121& 165${}^*$ \\ \hline
4                         & 0    & 0   & 3   & 7    & 12   & 18   & 34  &  &\\ \hline
5                         & 0    & 0   & 0   & 0    & 4    &   9   &     &  &  \\ \hline
\end{tabular}
}
\caption{Exact values of $\nu_k(K_n)$ for small values of $k$ and $n$,
  as computed using the maximum satisfiability solver {\tt
    Akmaxsat}. ${}^*$The value $\nu_3(K_{15})=165$ was not determined using {\tt
    Akmaxsat}; see Proposition~\ref{pro:315}.
\label{tab:max sat results} }
\end{center}
\end{table}
Each entry in  Table \ref{tab:max sat results} required at most 48 hours of computation on a laptop with 2.5GHz dual core processor and 4GB RAM;
 the values that are missing from the table could not be computed using  {\tt Akmaxsat} within this time.

We finally note that to our knowledge, prior to this work, the exact value of
$\nu_k(K_n)$ was not known for any $n,k$ with $2 < k
< \ceil{n/2}$ (we recall that $K_n$ can be drawn without crossings in
$\ceil{n/2}$ pages; thus $\nu_{k}(K_n) = 0$ for $k \ge \ceil{n/2}$ and
$\nu_k(K_n) > 0$ for $k < \ceil{n/2}$).

Some preliminary computation work was done using Sage \cite{SAGE}.

\subsection{Asymptotic bounds}
As an example of the numerical results presented here, let $m=69$ and
$k=10$. We computed $\mathcal{FJ}_{10}(G_{69})\approx 856,520$, and using this value
we get
$$\frac{\nu_{10}(K_{69})}{\binom{69}{4}}
                                   \ge     \frac{\binom{69}{4} - \mathcal{FJ}_{10}(G_{69})}{\frac{69}{4}} \\
                                  \approx  9.2313 \times 10^{-3}.
$$

Recall that, for all $n > m \ge 4$,
$$\frac{\nu_k(K_n)}{ \binom{n}{4}} \ge   \frac{\nu_k(K_m)}{\binom{m}{4}}.$$

Thus it follows that
$$\frac{\nu_{10}(K_n)}{ \binom{n}{4}} \ge   \frac{\nu_{10}(K_{69})}{\frac{69}{4}} \ge 
9.2313 \times 10^{-3} \quad (n > 69).$$
For $n > 69$, this is an improvement on the best previously known lower bound (from
\eqref{eq:nuk lower bound}), namely
\[
\frac{\nu_{10}(K_n)}{ \binom{n}{4}} \ge
\frac{2}{(3(10)-2)^2} \approx 2.5510\times 10^{-3}.
\]

In Table \ref{tab:JF bounds} we give a systematic
list of such improved bounds.
Computation was done on a Dell Precision T7500 workstation with 92GB of RAM memory, using the semidefinite
 programming solver SDPT3~\cite{SDPT3-ref1,SDPT3-ref2} under Matlab 7 together with the Matlab package YALMIP~\cite{YALMIP}.

%

\begin{table}[h!]
\begin{center}
{\footnotesize
\begin{tabular}{|c||c|c|c|c||c|} \hline
\diagbox{$k$}{$m$} & 39 & 49 & 59 & 69  & Lower bounds from (\ref{eq:nuk lower bound})  \\ \hline\hline
3 &  $1.4266\times 10^{-1}$ & $1.4827\times 10^{-1}$ & $1.5194\times 10^{-1}$ & $1.5452\times 10^{-1}$  & $ 2.0000\times 10^{-2}$    \\ \hline
4 &  $7.4205\times 10^{-2}$ & $7.9473\times 10^{-2}$ & $8.2837\times 10^{-2}$ & $8.5127\times 10^{-2}$  & $ 2.5210\times 10^{-2}$    \\ \hline
5 &  $4.2208\times 10^{-2}$ & $4.6916\times 10^{-2}$ & $5.0019\times 10^{-2}$ & $5.2141\times 10^{-2}$  & $ 7.8125\times 10^{-3}$    \\ \hline
6 &  $2.5728\times 10^{-2}$ & $2.9633\times 10^{-2}$ & $3.2258\times 10^{-2}$ & $3.4151\times 10^{-2}$  & $ 7.8125\times 10^{-3}$    \\ \hline
7 &  $1.6260\times 10^{-2}$ & $1.9605\times 10^{-2}$ & $2.1895\times 10^{-2}$ & $2.3524\times 10^{-2}$  & $ 4.1322\times 10^{-3}$    \\ \hline
8 &  $1.0544\times 10^{-2}$ & $1.3390\times 10^{-2}$ & $1.5356\times 10^{-2}$ & $1.6812\times 10^{-2}$  & $ 4.1322\times 10^{-3}$    \\ \hline
9 &  $6.9603\times 10^{-3}$ & $9.3377\times 10^{-3}$ & $1.1062\times 10^{-2}$ & $1.2333\times 10^{-2}$  & $ 2.5510\times 10^{-3}$    \\ \hline
10&  $4.6086\times 10^{-3}$ & $6.6189\times 10^{-3}$ & $8.1148\times 10^{-3}$ & $9.2314\times 10^{-3}$  & $ 2.5510\times 10^{-3}$    \\ \hline
11&  $3.0659\times 10^{-3}$ & $4.7436\times 10^{-3}$ & $6.0329\times 10^{-3}$ & $7.0285\times 10^{-3}$  & $ 1.7301\times 10^{-3}$    \\ \hline
12&  $2.0007\times 10^{-3}$ & $3.4078\times 10^{-3}$ & $4.5294\times 10^{-3}$ & $5.3894\times 10^{-3}$  & $ 1.7301\times 10^{-3}$    \\ \hline
13&  $1.2987\times 10^{-3}$ & $2.4613\times 10^{-3}$ & $3.4307\times 10^{-3}$ & $4.2025\times 10^{-3}$  & $ 1.2500\times 10^{-3}$    \\ \hline
14&  $8.2096\times 10^{-4}$ & $1.7736\times 10^{-3}$ & $2.6077\times 10^{-3}$ & $3.2930\times 10^{-3}$  & $ 1.2500\times 10^{-3}$    \\ \hline
15&  $4.7807\times 10^{-4}$ & $1.2613\times 10^{-3}$ & $1.9718\times 10^{-3}$ & $2.5870\times 10^{-3}$  & $ 9.4518\times 10^{-4}$    \\ \hline
16&  $2.6556\times 10^{-4}$ & $8.9554\times 10^{-4}$ & $1.5141\times 10^{-3}$ & $2.0348\times 10^{-3}$  & $ 9.4518\times 10^{-4}$    \\ \hline
17&  $1.3191\times 10^{-4}$ & $6.2938\times 10^{-4}$ & $1.1514\times 10^{-3}$ & $1.6023\times 10^{-3}$  & $ 7.3964\times 10^{-4}$    \\ \hline
18&  $5.2726\times 10^{-5}$ & $4.2802\times 10^{-4}$ & $8.5199\times 10^{-4}$ & $1.2562\times 10^{-3}$  & $ 7.3964\times 10^{-4}$    \\ \hline
19&  $8.8699\times 10^{-6}$ & $2.7320\times 10^{-4}$ & $6.3294\times 10^{-4}$ & $9.8258\times 10^{-4}$  & $ 5.9453\times 10^{-4}$    \\ \hline
20&  $0$                    & $1.7127\times 10^{-4}$ & $4.7985\times 10^{-4}$ & $7.7482\times 10^{-4}$  & $ 5.9453\times 10^{-4}$    \\ \hline
\end{tabular}
}
\caption{Lower bounds for $\frac{\nu_k(K_n)}{\binom{n}{4}} \ge \frac{\binom{m}{4} 
- \mathcal{FJ}_{k}(G_{m})}{\binom{m}{4}}$, for all $n
  > m$, $m \in\{39,49,59,69\}$ and $k = 3,4,\ldots,20$,
and comparison with the previous best lower bounds on $\frac{\nu_k(K_n)}{\binom{n}{4}}$ (from (\ref{eq:nuk lower bound})). \label{tab:JF bounds}}.
\end{center}
\end{table}

Note that the values in the column ``$m=69$'' improve on the known lower bounds  (\ref{eq:nuk lower bound})  in all cases, for $n > 69$.

\begin{table}[h!]
\begin{center}
{\footnotesize
\begin{tabular}{|c|c|c|c|c|}\hline
& & & & \\[3pt]
 & Previous lower bound  & Improved lower bound   &
Best upper bound & Quotient between lower  \\
$k$ & on $\lim_{n\to\infty}
\frac{\nu_k(K_n)}{{\binom{n}{4}}}$ &
on $\lim_{n\to\infty}
 \frac{\nu_k(K_n)}{{\binom{n}{4}}}$ &
on $\lim_{n\to\infty}
  \frac{\nu_k(K_n)}{{\binom{n}{4}}}$ & and upper bound \\[0.4cm]
\hline
3 & $ 2.0000\times 10^{-2}$ &  $1.5452 \times 10^{-1}$ &
$1.8518\times 10^{-1}$ & $0.8344$      \\ \hline
4 & $ 2.5210\times 10^{-2}$ & $8.5127\times 10^{-2}$  & $1.0937\times 10^{-1}$  &  $0.7783$   \\ \hline
5 & $ 7.8125\times 10^{-2}$ & $5.2141\times 10^{-2}$  & $7.2000\times 10^{-2}$  &  $0.7241$   \\ \hline
6 & $7.8125\times 10^{-3}$ & $3.4151\times 10^{-2}$ & $5.0925\times 10^{-2} $  &  $0.6706$   \\ \hline
7 & $ 4.1322\times 10^{-3}$ & $2.3524\times 10^{-2}$  &  $3.7900\times
10^{-2}$& $0.6706$    \\ \hline
8 & $ 4.1322\times 10^{-3}$ & $1.6812\times 10^{-2}$  &  $2.9296\times
10^{-2}$&  $0.5738$   \\ \hline
9 & $ 2.5510\times 10^{-3}$ & $1.2333\times 10^{-2}$  & $2.3319\times 10^{-2}$  & $0.5287$    \\ \hline
10& $ 2.5510\times 10^{-3}$ & $9.2314\times 10^{-3}$  & $1.9000\times 10^{-2}$  & $0.4858$    \\ \hline
11& $ 1.7301\times 10^{-3}$ & $7.0285\times 10^{-3}$  & $1.5777\times 10^{-2}$  & $0.4454$    \\ \hline
12& $ 1.7301\times 10^{-3}$ & $5.3894\times 10^{-3}$  & $1.3310\times 10^{-2}$  & $0.4049$    \\ \hline
13& $ 1.2500\times 10^{-3}$ & $4.2025\times 10^{-3}$  & $1.1379\times 10^{-2}$  & $0.3693$    \\ \hline
14& $ 1.2500\times 10^{-3}$ & $3.2930\times 10^{-3}$  & $9.8396\times 10^{-3}$  & $0.3346$    \\ \hline
15& $ 9.4518\times 10^{-3}$ & $2.5870\times 10^{-3}$  & $8.5925\times 10^{-3}$  & $0.3010$    \\ \hline
16& $ 9.4518\times 10^{-4}$ & $2.0348\times 10^{-3}$  & $7.5683\times 10^{-3}$  & $0.2688$    \\ \hline
17& $ 7.3964\times 10^{-4}$ & $1.6023\times 10^{-3}$  & $6.7168\times 10^{-3}$  & $0.2385$    \\ \hline
18& $ 7.3964\times 10^{-4}$ & $1.2562\times 10^{-3}$  & $6.0013\times 10^{-3}$  & $0.2093$    \\ \hline
19& $ 5.9453\times 10^{-4}$ & $9.8258\times 10^{-4}$  & $5.3943\times 10^{-3}$  & $0.1821$    \\ \hline
20& $ 5.9453\times 10^{-4}$ & $7.7482\times 10^{-4}$  & $4.8750\times
10^{-3}$  & $0.1589$
\\ \hline
\end{tabular}
}
\caption{\small Summary of lower and upper bounds for
  $\lim_{n\to\infty} \nu_k(K_n)/\binom{n}{4}$. The second column gives
  the previously best lower bounds,  as given in \eqref{eq:forfour}, \eqref{eq:foreven},
  and \eqref{eq:forodd}.  The third column presents the lower bounds we obtained
  by computing $\mathcal{FJ}_k(G_{69})$ for $k=3,4,\ldots,20$ (this is
  the fifth column of Table~\ref{tab:JF bounds}). In the fourth column
  we show the best upper bounds known, given by \eqref{eq:nuk upper bound}
 (alternatively, using Observation~\ref{obs:kdin2} and that
  $\nu_k(K_n) \le Z_k(n)$). Finally, in the fifth column we show the ratio
  between the values given in the third and fourth columns.
}
\label{tab:newtab}
\end{center}
\end{table}

In Table~\ref{tab:newtab} we summarize the best 
lower and upper bounds known for $\lim_{n\to\infty}
\nu_k(K_n)/\binom{n}{4}$.

\section{Drawing $K_n$ in $k$ pages: conjectures and results}\label{sec:upperbounds}

In this section we calculate the number $Z_k(n)$ of crossings that result by
drawing $K_n$ on $k$ pages using the construction by Damiani et
al.~\cite{DaDASa} (we recall that we call this the {\em DDS construction}), and
a generating function $G_k(z):=\sum_{n\geq 0} Z_k(n) z^n$ for it.
This construction is a natural generalization of the construction by
Bla\v{z}ek and Koman~\cite{bk} (who fully described it for $2$ pages,
and briefly mentioned that it could be generalized to $k>2$ pages), and a slight refinement of the
construction by Shahrokhi et al.~\cite{sssv96}. 
The description of the construction and the calculation of $Z_k(n)$ and $G_k(z)$ are
in Section~\ref{sub:cons}.

We calculate the exact value of
$Z_k(n)$ for two reasons. First, the value $Z_k(n)$ was
determined in neither~\cite{bk} nor~\cite{DaDASa}; since no
better (crossing-wise) construction to draw $K_n$ in $k$ pages is
known, this is a calculation worth doing.  Second, for all values of
$k$ and $n$ for which we now (that is, with the results reported in
this paper)  know the exact value of $\nu_k(K_n)$, we have $\nu_k(K_n)=Z_k(n)$.
These confirmations, as well as an additional feature that we shall
explain below (namely Proposition~\ref{pro:oddeven}, from which we
will compute $\nu_3(K_{15})$), 
lend credibility to the conjecture
$\nu_k(K_n)=Z_k(n)$, which we formally put forward in 
Section~\ref{sub:gencon}. As we shall see, the value of $Z_k(n)$
depends on $n$\,mod\,$k$, and for each fixed $k$ and each
fixed $q\in\{0,1,\ldots,k-1\}$, there is a degree $4$ polynomial
$G_{q,k}(n)$ such that $Z_k(n)=G_{q,k}(n)$ for all $n$ such that 
$n$\,mod\,$k =q$.
In Section~\ref{sub:kdividesn} we explicitly give this polynomial for the case
$n$\,mod\,$k=0$.
We finally present, in Section~\ref{sub:keq3},  a slightly more
detailed discussion and further results for the case $k=3$.


\subsection{The DDS construction: an upper bound $Z_k(n)$ for
  the $k$-page \\ crossing number
  of $K_n$}\label{sub:cons}

The DDS construction was described in~\cite{DaDASa} in terms of the adjacency
matrix.  We have found it both more lively and  more convenient
(for our calculations) to follow the more geometrical viewpoint of
Shahrokhi et al.~to describe this construction, and this is the
approach we follow below.

We draw $K_n$ in $k$ pages using the circular model. Label the
vertices $0,1,2,\ldots,n-1$ in the clockwise order in which they occur
in the boundary of the circle. 
For $i=0,1,\ldots,n-1$, let $M_i$ be the set of edges whose endpoints
have sum $i$ (modulo $n$). Thus, $M_i$ is a matching for each
$i\in\{0,1,\ldots,n-1\}$. We note that each edge belongs to exactly
one matching $M_i$. 
For $s,t\in\{0,1,\ldots,n-1\}$, $s<t$, let
$\M_{s,t}:=M_s\cup M_{s+1}\cup \cdots M_{t}$.  Loosely speaking,
$\M_{s,t}$ consists of the edges of $t-s+1$ ``consecutive''
matchings. In Figure~\ref{fig:k103pages} we illustrate the sets
$\M_{0,3}$ (left), $\M_{4,6}$ (center), and $\M_{7,9}$ (right), for
the case $n=10$.

\begin{figure}[h!]
\begin{center}
\scalebox{0.28}{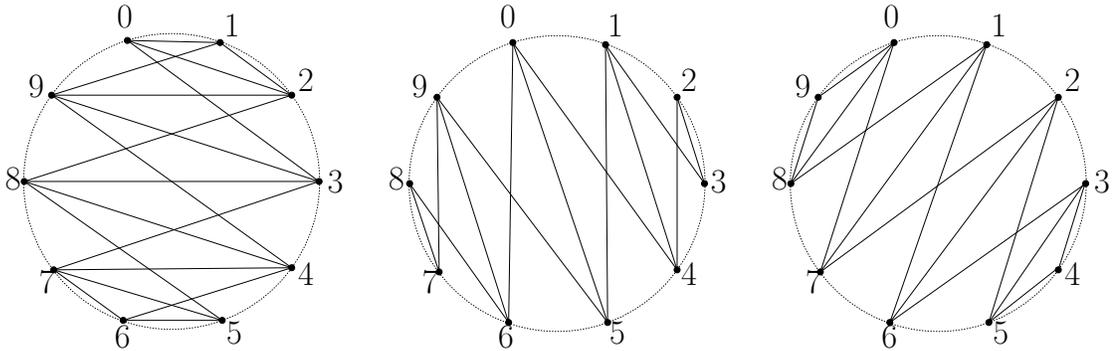}
\caption{\label{fig:k103pages} To draw $K_{10}$ in $3$ pages, we place
  the edges in $\M_{0,3}=M_0\cup M_1\cup M_2\cup M_3$ in page $0$
  (left), the
  edges in $\M_{4,6}=M_4\cup M_5\cup M_6$ in page $1$ (center), and the edges
  in $M_{7,9}=M_7\cup M_8\cup M_9$ in page $2$ (right).
}
\end{center}
\end{figure}

Let $p:=\floor{n/k}$ and $q:=n$ mod $k$ (thus
$n=pk+q$). The DDS construction consists simply on
distributing the edges of $K_n$ into $k$ pages $0,1,\ldots,k-1$ 
as follows:

\begin{enumerate}
\item\label{it:twop}
for
$0 \le \ell < q$, place in page $\ell$  the edges in 
$\M_{\ell(p+1), \ell (p+1) +p }$; and
\item\label{it:onep} for $q \le \ell < k$, place in page $\ell$ the edges in 
$\M_{\ell p + q, \ell p + q + (p-1)}$.
\end{enumerate}

Thus, if $0\le \ell < q$, then page $\ell$ contains the edges of $p+1$
matchings, and if $q \le \ell < k$, then page $\ell$ contains the
edges of $p$ matchings. 
Note that if $k$ divides $n$ (that is, $q=0$), then there is no $\ell$
such that $ 0 \le \ell < q$, and so each page contains the edges
of $p$ matchings.  In Figure~\ref{fig:k103pages} we illustrate the
DDS construction for the case $k=3, n=10$.

We shall give $Z_k(n)$ in terms of a function $F$ that we now
define. First, let
\[
f(r) := 
\frac{rn}{2}  - \frac{r^2}{2} -\frac{n}{2}+\frac{1}{2},
\]
and then let 
\begin{equation}\label{eq:def1}
F(r,n):=
\sum_{\ell=0}^{r-1} (r-\ell)f(\ell) = 
-\frac{r^4}{24}+{\frac{n r^3}{12}}-{\frac{n r^2}{4}}+{\frac{7 r^2 }{24}}+{\frac{n r}{6}}-{\frac{r}{4}}.
\end{equation}

\begin{proposition}\label{pro:calcu}
The number of crossings that result from drawing $K_n$ in $k\ge 1$
pages using the DDS construction is 
\[
Z_k(n):=\nmodk{n}{k}\cdot
F\biggl(\bigfloor{\frac{n}{k}}+1,n\biggl) +
\bigl(k - \nmodk{n}{k}\bigr) \cdot F\biggl(\bigfloor{\frac{n}{k}},n\biggr).
\]
Thus $K_n$ can be drawn in $k$ pages with $Z_k(n)$ crossings, and so
\[
\nu_k(K_n) \le Z_k(n).
\]
\end{proposition}

Note that $Z_k(n)$ is a quasi-polynomial of period $k$ in $n$ 
(cf. e.g. R.~Stanley~\cite[Sect.~4.4]{rstI}). This implies that its generating function $G_k(z)$
is rational, i.e. the ratio of two polynomials in $z$, with denominator having only $k$-th roots of unity
as roots.
We will calculate the generating function $G_k(z)$ for $Z_k(n)$ below
(cf.~Proposition~\ref{pro:gfZ_k}).

\begin{proof}[Proof of Proposition~\ref{pro:calcu}]
Suppose first that $k=1$. All the edges are then drawn in the same
page. Thus every four points define a crossing, and so we have
$\binom{n}{4}$ crossings in total. Since
\[Z_1(n) = F(n,n)= 
-{\frac{n^4}{24}}+{\frac{n^4}{12}}-{\frac{n^3}{4}}+
{\frac{7 n^2 }{24}}+{\frac{n^2}{6}}-{\frac{n}{4}}= \binom{n}{4},
\]
it follows that the statement is true for $k=1$.

Thus we suppose for the rest of the proof that $k\ge 2$.

To calculate the number of crossings in each page, we first need to
calculate the crossings between edges in distinct matchings, when
these matchings are placed in the same page.
If $M_i$ and $M_j$ are in
the same page, then the number 
$\CM{M_i}{M_j}$ of crossings
involving an edge in $M_i$ and an edge in $M_j$ depends 
on the parity
of $i,j$, and $n$, as well as on $j-i$.
It is an easy exercise to show that for all $i,j$ such that
$0 \le i < j \le n-1$ and $j-i \le n/2$,
\begin{equation}
\CM{M_i}{M_j}=
\begin{cases}
f(j-i), &\text{ if $n$ is odd};\\
f(j-i) -\frac{1}{2}, &\text{ if $i$ and $
  j$
  are odd and $n$ is even}; \\
f(j-i),
&\text{ if $i$ and $j$ have distinct parity and $n$ is even;} \\
f(j-i) + \frac{1}{2}, 
&\text{ if
  $i, j$, and $n$ are even.} \\
\end{cases}
\end{equation}


With this information at hand, we may proceed to calculate the number
$\ucr(\M_{s,t})$ of crossings with both edges in $\M_{s,t}$, when all
the edges in $\M_{s,t}$ are in the same page. Formally, for $s,t$ such
that $0 \le s < t \le n-1$, let
$\ucr(\M_{s,t}) := \sum_{s \le i < j \le t}
\CM{M_i}{M_j}$. Our aim (as
this is all we shall need) is
to calculate $\ucr(\M_{s,t})$ for values of $s$ and $t$ such that $0
\le s < t \le n-1$ and $t-s \le n/2$.

Note that there are two types of collections
$\M_{s,t}$ that appear in the construction: those of the form 
$\M_{\ell(p+1), \ell (p+1) +p }$ for $\ell\in\{0,1,\ldots,q-1\}$
(these contain the edges in $p$ matchings, and we call them {\em large}
collections), and those of the form 
$\M_{\ell p + q, \ell p + q + (p-1)}$ for
$\ell\in\{q,q+1,\ldots,k-1\}$ (these contain the edges in $p-1$ matchings, and we
call them {\em small} collections). Thus there are $q$ large
collections and $k-q$ small collections.

We observe that it follows immediately from the construction that
\begin{equation}\label{eq:thesu}
Z_k(n)= 
\sum_{\M_{s,t} \text{ large }} \ucr(\M_{s,t}) +
\sum_{\M_{s,t} \text{ small }} \ucr(\M_{s,t}),
\end{equation}
where the first summation is over all collections  $\M_{s,t}$ in the 
construction that are
large, and the second summation is over all collections $\M_{s,t}$ in
the construction
that are small.

We will analyze separately the two possibilities for the parity of
$n$.

\vglue 0.3 cm
\noindent{\sc Case 1.} {\em $n$ is odd}
\vglue 0.3 cm

Let $n$ be odd, and let 
$s,t$ satisfy $0 \le s < t \le n-1$ and $t-s  \le n/2$. 

Then
\begin{align*}
\ucr(\M_{s,t}) &= 
\sum_{s \le i < j \le t} \CM{M_i}{M_j} 
= 
\sum_{{s \le i < j \le t} } f(j-i) 
= \sum_{1\le\ell\le t-s} ((t-s+1)-\ell)f(\ell)\\
&=
\sum_{1\le\ell\le t-s} ((t-s+1)-\ell)
\biggl(\frac{\ell n}{2}  - \frac{\ell^2}{2} -\frac{n}{2}+\frac{1}{2}\biggr) 
=
F(t-s+1,n).
\end{align*}


Thus it follows that if $\M_{s,t}$ is a large collection, then
$\ucr(\M_{s,t}) = F(t-s+1,n)=F(p+1,n)$, and if it is small, then
$\ucr(\M_{s,t}) = F(t-s+1,n)=F(p,n)$. Using this 
and \eqref{eq:thesu}, and recalling that there are $q=n\,\text{mod }k$
large collections and $k-q$ small collections, and 
that $p=\floor{n/k}$, 
we obtain
\begin{align*}
Z_k(n)&= q\cdot
F(p+1,n) + (k-q)\cdot 
F(p,n)\\
&=
\nmodk{n}{k}\cdot F\biggl(\bigfloor{\frac{n}{k}}+1,n\biggr) + (k-\nmodk{n}{k}) 
\cdot F\biggl(\bigfloor{\frac{n}{k}},n\biggr).
\end{align*}

\vglue 0.3 cm
\noindent{\sc Case 2.} {\em $n$ is even}
\vglue 0.3 cm

Let $n$ be even, and let 
$s,t$ satisfy $0 \le s < t \le n-1$ and $t-s  \le n/2$.
For this case ($n$ even), the determination of 
$\ucr(\M_{s,t})$ is more involved, since it
depends both on the parity of $t-s$ and on the
parity of $s$. 
for $i <r$; and
$\rr_i= \{ R_{im+r},R_{im+r+1},\ldots,R_{(i+1)m+(r-1)} \}$ for $r \le i < k$.

To simplify the expressions it is convenient to define 
\begin{align*}
\oddodd{s}{t} &:= \bigl| \{ (i,j) \ | \ s \le i < j \le t, \hbox{$i$ odd,
  $j$ odd}\} \bigr|, \\
\oddeven{s}{t} &:= \bigl| \{ (i,j) \ | \ s \le i < j \le t, \hbox{$i$ odd,
  $j$ even}\} \bigr|, \\
\evenodd{s}{t} &:= \bigl| \{ (i,j) \ | \ s \le i < j \le t, \hbox{$i$ even,
  $j$ odd}\}\bigr|, \\
\eveneven{s}{t} &:= \bigl| \{ (i,j) \ | \ s \le i < j \le t, \hbox{$i$ even,
  $j$ even}\}\bigr|.
\end{align*}

An elementary argument shows that
\begin{equation}\label{eq:theeo}
\eveneven{s}{t}\, - \oddodd{s}{t} = 
\begin{cases}
\frac{t-s}{4}, &\text{if both $s$ and $t$ are even;}\\
0, &\text{if $s$ and $t$ have distinct parity; and}\\
-\frac{t-s}{4}, &\text{if both $s$ and $t$ are odd}.
\end{cases}
\end{equation}

We have
\begin{align*}
\ucr(\M_{s,t}) &= 
\sum_{s \le i < j \le t} \CM{M_i}{M_j} 
 = 
\sum_{\stackrel{s \le i < j \le t}{i \eve} } \CM{M_i}{M_j}
+
\sum_{\stackrel{s \le i < j \le t}{i \odd} } \CM{M_i}{M_j}\\
&=
\sum_{\stackrel{s \le i < j \le t}{i \eve,j \eve} } \CM{M_i}{M_j}
+
\sum_{\stackrel{s \le i < j \le t}{i \eve,j \odd} } \CM{M_i}{M_j}\\
&+
\sum_{\stackrel{s \le i < j \le t}{i \odd, j \eve} } \CM{M_i}{M_j}
+
\sum_{\stackrel{s \le i < j \le t}{i \odd, j \odd} } \CM{M_i}{M_j}\\
&= 
\sum_{\stackrel{s \le i < j \le t}{i \eve,j \eve} } \bigl(f(j-i)+\frac{1}{2}\bigr)
+
\sum_{\stackrel{s \le i < j \le t}{i \eve,j \odd} } f(j-i)\\
&+
\sum_{\stackrel{s \le i < j \le t}{i \odd, j \eve} } f(j-i)
+
\sum_{\stackrel{s \le i < j \le t}{i \odd, j \odd} } \bigl(f(j-i)-\frac{1}{2}\bigr) \\
&=
\sum_{s \le i < j \le t} f(j-i)
+\sum_{\stackrel{s \le i < j \le t}{i \eve,j \eve} } \frac{1}{2}
\,\,\,\,-\sum_{\stackrel{s \le i < j \le t}{i \odd,j \odd} } \frac{1}{2}\\
&=
\sum_{s \le i < j \le t} f(j-i)
+\frac{1}{2} \eveneven{s}{t} - \frac{1}{2} \oddodd{s}{t} \\
&=
\sum_{1\le\ell\le t-s} ((t-s+1)-\ell)f(\ell) + 
\frac{1}{2} \eveneven{s}{t} - \frac{1}{2} \oddodd{s}{t} \\
&=
 F(t-s+1,n) + \frac{1}{2} \eveneven{s}{t} - \frac{1}{2} \oddodd{s}{t}.
\end{align*}


Using this last expression and \eqref{eq:theeo}, it follows that
\begin{equation}\label{eq:thero}
\ucr(\M_{s,t}) = 
\begin{cases}
F(t-s+1,n) + \frac{t-s}{4}, &\text{if $s$ and $t$ are even};\\
F(t-s+1,n), &\text{if $s$ and $t$ have distinct parity};\\
F(t-s+1,n) - \frac{t-s}{4}, &\text{if $s$ and $t$ are odd}.
\end{cases}
\end{equation}

Let us say that a collection $\M_{s,t}$ is {\em even-odd} if $s$ is even and $t$
is odd; it is {\em even-even} if $s$ and $t$ are even; it is {\em
  odd-even} if $s$ is odd and $t$ is even; and it is {\em odd-odd} if
$s$ and $t$ are odd.

We now proceed to compute $Z_k(n)$, analyzing separately the
cases when $p$ is odd and when $p$ is even.

Suppose first that $p$ is odd. It is readily verified that in this
case (i) each of the $q$ large collections is either even-odd or
odd-even; and (ii) out of the $k-q$ small collections,
$(k-q)/2$ are even-even and $(k-q)/2$ are odd-odd. 
Recalling that if a collection $\M_{s,t}$ is large then $t-s=p$ and
that if it is small then $t-s=p-1$, and using
\eqref{eq:thesu} and \eqref{eq:thero}, it follows that
\begin{align*}
Z_k(n)&=q\cdot F(p+1,n) + 
\frac{(k-q)}{2}\cdot \biggl(F(p,n) + \frac{(p-1)}{4}\biggr)+
\frac{(k-q)}{2}\cdot \biggl(F(p,n) - \frac{(p-1)}{4}\biggr)\\
&=
q\cdot F(p+1,n) + (k-q)\cdot F(p,n)\\
&= 
\nmodk{n}{k}\cdot F\biggl(\bigfloor{\frac{n}{k}}+1,n\biggr) + (k-\nmodk{n}{k}) 
\cdot F\biggl(\bigfloor{\frac{n}{k}},n\biggr).
\end{align*}

Suppose finally that $p$ is even.
It is easily checked that in this case 
(i) out of the $q$ large collections, $q/2$ are even-even and $q/2$
are odd-odd; and (ii)
each of the $k-q$ small collections is either even-odd or odd-even.
Recalling again that if a collection $\M_{s,t}$ is large then $t-s=p$ and
that if it is small then $t-s=p-1$, and using
\eqref{eq:thesu} and \eqref{eq:thero}, it follows that
\begin{align*}
Z_k(n)&=\frac{q}{2}\cdot \biggl(F(p+1,n) + \frac{p}{4}\biggr) +
\frac{q}{2}\cdot \biggl(F(p+1,n) - \frac{p}{4}\biggr)+
(k-q)\cdot F(p,n)\\
&=
q\cdot F(p+1,n) + (k-q)\cdot F(p,n)\\
&= 
\nmodk{n}{k}\cdot F\biggl(\bigfloor{\frac{n}{k}}+1,n\biggr) + (k-\nmodk{n}{k}) 
\cdot F\biggl(\bigfloor{\frac{n}{k}},n\biggr). \qedhere
\end{align*}
\end{proof}


\begin{proposition}\label{pro:gfZ_k}
For a fixed $k$, the generating function $G_k(z)$ for $Z_k(n)$ is 
\[
G_k(z):=\sum_{n\geq 0} Z_k(n) z^n=z^{2k+1}\frac{(k-2)(1-z)+1-z^{k+1}}{(1-z)^3(1-z^k)^3}.
\]
\end{proposition}

As a first application of this formula, one  sees at once that $Z_k(n)=0$ 
for $n\leq 2k$, as the first nonzero coefficient in the
expansion of $G_k$ into powers of $z$ comes up for the $2k+1$-th power.

\begin{proof}[A sketch of proof of Proposition~\ref{pro:gfZ_k}]
Note that 
\begin{equation}\label{eq:G_k}
G_k(z)=\sum_{s\geq 0}z^{sk}\sum_{\rho=0}^{k-1}Z_k(sk+\rho)z^\rho,
\end{equation} 
and
in this form one does not have to worry about $n \mod k$ and $\lfloor\frac{n}{k}\rfloor$,
as $Z_k(sk+\rho)$ is a polynomial in $s$ and $\rho$.
One computes the inner sum in \eqref{eq:G_k}
to see that it is equal to an explicit degree $4$ polynomial in $s$ divided by $(z-1)^3$,
namely, 
\begin{multline*}
\frac{24(z-1)^3}
{ s(s - 1) } \sum_{\rho=0}^{k-1}Z_k(sk+\rho)z^\rho=
(2 k^2 s^2 + 2 k^2 - 13 k s - 12 k + 4 s + 16 - k s^2 + 4 k^2 s) z^{2+k}\\ +
( - 4 k^2 s^2 - 8 k^2 s + 2 k s^2 - 4 k^2 + 18 k s - 4 s - 4 + 16 k)z^{1+k} \\
+ (2 k s^2 + 4 k s -  s^2 + 2 k - 5  s - 4 )k z^k 
+ (4 k^2 s - 2 k^2 s^2 + 9 k s + k s^2 - 10 k - 4 s - 16) z^2\\ 
+ (4 k^2 s^2 - 8 k^2 s - 2 k s^2 - 10 k s + 16 k + 4 s + 4) z
- 2 k^2 s^2 + 4 k^2 s + k s^2 + k s - 6 k.
\end{multline*}
It remains to observe that the outer sum in  \eqref{eq:G_k}
becomes a finite sum of terms of the 
form $C\sum_{s\geq 0}s^\ell z^s$, with $C$ independent of $s$. A direct computation then
gives the claimed formula.
\end{proof}

\subsection{A conjecture for the $k$-page crossing number of $K_n$}\label{sub:gencon}

The DDS construction described in Section~\ref{sub:cons} draws $K_n$ in $k$ pages
with $Z_k(n)$ crossings. We conjecture the optimality of this
construction:

\begin{conjecture}\label{con:kn}
For all positive integers $k$ and $n$, 
\[
 \nu_k(K_n) = Z_k(n).
\]
\end{conjecture}

The naturality and aesthetical appeal of the construction, plus the
fact that no construction to draw $K_n$ in $k$ pages with fewer
crossings is known, seem good enough reasons to put forward this
conjecture. Still, there is further evidence supporting the
conjecture:

\begin{itemize}
\item The statement is true
for $k\le 2$. For $k=1$ this is readily checked, and for $k=2$ it
follows since
it has been recently verified that $\nu_2(K_n)=Z_2(n)$~\cite{abregoetal}.
\item For all $k,n$ for which we now know the exact value of $\nu_k(K_n)$
(Table~\ref{tab:max sat results}, plus all $k,n$ such that
  $k> \ceil{n/2}$, for which it is known that
$\nu_k(K_n)=0$), we have
$\nu_k(K_n)=Z_k(n)$.
\end{itemize}

There is yet another argument that supports Conjecture~\ref{con:kn},
at a somewhat (but not completely; see Section~\ref{sub:keq3}) more speculative level.  Recall that
$Z_2(n):=\frac{1}{4}\floor{\frac{n}{2}}\floor{\frac{n-1}{2}}
\floor{\frac{n-2}{2}}\floor{\frac{n-3}{2}}$.  A well-known counting
argument shows that for every positive integer $r$, $\nu_2(K_{2r-1}) =
Z_2(2r+1)$ implies $\nu_2(K_{2r})=Z_2(2,2r)$. This
``odd implies even'' phenomenon is used, for instance, to determine
that the (usual) crossing number of $K_{12}$ is $Z_2(12)$: this follows
at once since the crossing number of $K_{11}$ is
$Z_2(11)$~\cite{panrichter}. An appealing feature of
Conjecture~\ref{con:3p} is that it implies a similar phenomenon for
every $k$:

\begin{proposition}\label{pro:oddeven}
For every positive integers $k$ and $r$, one has $krZ_k(kr-1)=(kr-4)Z_k(kr)$, and
\[
\nu_k(K_{kr-1}) = Z_k(kr-1) \text{\hglue 0.3
  cm}\implies
\text{\hglue 0.3 cm} \nu_k(K_{kr}) = Z_k(kr).
\]
\end{proposition}

\begin{proof}
The first claim follows from 
\begin{equation*}
{ Z_k(kr-1)} =
{(k-1)F(r,kr-1) + F(r-1,kr-1)}
=\frac{kr-4}{r} F(r,kr)=\frac{kr-4}{kr}Z_k(kr),
\end{equation*}
after a long but routine manipulation.

Since $\nu_k(K_{kr}) \le Z_k(kr)$ (cf. Proposition~\ref{pro:calcu}), we only need
to prove the reverse inequality $\nu_k(K_{kr}) \ge Z_k(kr)$.
Suppose that $\nu_k(K_{kr-1}) = Z_k(kr-1)$.
Consider a $k$-page drawing $D$ of $K_{kr}$ with $\nu_k(K_{kr})$ crossings. This drawing contains
$kr$ drawings of $K_{kr-1}$, each of which has at least
$\nu_k(K_{kr-1}) = Z_k(kr-1)$ crossings. It is easy to
see that each crossing
gets counted exactly $kr-4$ times, and so
\begin{equation*}
\nu_k(K_{kr}) \ge \frac{kr Z_k(kr-1)}{kr-4}=k F(r,kr)=Z_k(kr). \qedhere
\end{equation*}
\end{proof}

In Section~\ref{sub:keq3} we will use Proposition~\ref{pro:oddeven} to
prove that $\nu_3(K_{15}) = 165$ (cf.~Proposition~\ref{pro:315}).


\subsection{Explicit estimates for $Z_k(n)$}\label{sub:kdividesn}

It is
clear that for each fixed $k$ and $q\in \{0,1,\ldots,k-1\}$,
there exists a polynomial $G_{k,q}(n)$ such that $G_{k,q}(n)=Z_k(n)$ for
all $n$ such that $n$\,mod\,$k=q$. For the case $q=0$, a 
routine manipulation yields the following.

\begin{observation}\label{obs:kdin} If $k$ divides $n$, then
\[
Z_k(n) =\biggl( \biggl(\frac{1}{12k^2}\biggr)\biggl(1-\frac{1}{2k}\biggr)\biggr) n^4 +
\biggl( -\frac{1}{4k}\biggr) n^3 + \biggl( \frac{7}{24k} +
\frac{1}{6}\biggr) n^2 +
\nonumber \biggl( -\frac{1}{4}\biggr) n.
\]
\end{observation}


We recall (see Section~\ref{sub:lus}) that $\nu_k(K_n)/\binom{n}{4} \ge
\nu_k(K_m)/\binom{m}{4}$, whenever $n > m \ge 4$.  Using this and
Observation~\ref{obs:kdin}, we obtain the following asymptotic
general estimate for $Z_k(n)$:

\begin{observation}\label{obs:kdin2} 
For each positive integer $k$,\[
Z_k(n) = \biggl(
\biggl(\frac{1}{12k^2}\biggr)\biggl(1-\frac{1}{2k}\biggr)\biggr) n^4 +
O(n^3).
\]
\end{observation}

We note that the upper bound \eqref{eq:nuk upper bound}, given by
Shahrokhi et al.~\cite{sssv96}, is in line with this last observation 
(this was expected, since the DDS
construction and the construction in~\cite{sssv96} agree whenever $k$
divides $n$).

\subsection{Drawing $K_n$ in $3$ pages: further results}\label{sub:keq3}

A long but straightforward manipulation shows that
\begin{equation}\label{eq:k3gen}
Z_3(n)=
\begin{cases}
\frac{(n-6)(n-3)n(5n-9)}{648}, 
&\text{\rm \hglue
  0.2 cm if }
n\equiv 0\, (\text{\rm mod} \, 3);\\[0.2cm]
\frac{(n-4)(n-1)(5n^2 -29 n + 30)}{648},
 &\text{\rm \hglue
  0.2 cm if }
n\equiv 1\, (\text{\rm mod} \, 3);\\[0.2cm]
\frac{(n-2)(n-3)(n-5)(5n-4)}{648},
&\text{\rm \hglue
  0.2 cm if }
n\equiv 2\, (\text{\rm mod} \, 3).
\end{cases}
\end{equation}

We remark that this coincides with the number of crossings given by
Bla\v{z}ek and Koman in~\cite{bk}, where they briefly mentioned that
their construction for $2$ pages could be generalized to $k>2$ pages,
and reported the expression in~\eqref{eq:k3gen} for the number of
crossings obtained by drawing $K_n$ in $3$ pages
(no further details were given).

As we have already observed, $\nu_3(K_n)=Z_3(n)$ for all $8$ values
of $n\ge 7$ for which we have calculated $\nu_3(K_n)$
(Table~\ref{tab:max sat results}).
This evidence gives special credence to the case $k=3$ of Conjecture~\ref{con:kn}:

\begin{conjecture}[3-page crossing number of $K_n$]\label{con:3p}
For every positive integer $n$,
\[
 \nu_3(K_n) = Z_3(n).
\]
\end{conjecture}

As an additional support for this conjecture, we note that
our calculations reported in Section~\ref{sec:numerical results}
confirm that $\nu_3(K_n)$ is reasonably close to $Z_3(n)$, at least
asymptotically. Indeed, from Table~\ref{tab:newtab} we have
$\lim_{n\to\infty}\nu_3(K_n)/\binom{n}{4} \ge 0.15452$. This implies
that $\lim_{n\to\infty}\nu_3(K_n)/n^4 \ge 0.006438$. Now from
\eqref{eq:k3gen} we have $\lim_{n\to\infty}  Z_3(n)/n^4 =
5/648$. These results yield
\[
\lim_{n\to\infty} \frac{\nu_3(K_n)}{Z_3(n)}>
\frac{0.006438}{5/648} \approx 0.8344.
\]

We finally show that, as hinted above, 
Proposition~\ref{pro:oddeven} (the generalization to $k>2$ pages of
the ``odd implies even'' phenomenon for $k=2$) is not only a speculative curiosity:
we use this statement to determine the exact value of $\nu_3(K_{15})$:

\begin{proposition}\label{pro:315}
$\nu_3(K_{15}) = 165$.
\end{proposition}

\begin{proof}
It follows from Proposition~\ref{pro:oddeven}, using (from
our calculation, reported in Table~\ref{tab:max sat results}) that
$\nu_3(K_{14}) = Z_3(14)$.
\end{proof}

\section{Concluding remarks and open questions}\label{sec:concludingremarks}

De Klerk, Pasechnik and Warners~\cite{KPW} proved the lower bounds
$\alpha_k$ on the ratio
$\frac{\mbox{max-$k$-cut}(G)}{\mathcal{FJ}(G)}$ given in
Table~\ref{table:guarantees}.  These lower bounds may be used to
obtain upper bounds on $\nu_k(K_n)$ $(3 \le k\le 10)$, namely,
\[
\nu_k(K_n) \le |E_n| - \alpha_k \mathcal{FJ}(G_n),
\]
but
these bounds seem weaker than the upper bounds given by the best known drawings, based on computations for
$3 \le k \le 10$ and $n \le 69$.

\begin{table}[h!]
{\small
\begin{center}
\begin{tabular}{|c|c|c|c|c|c|c|c|c|c|}\hline
$k$:  & 3& 4& 5& 6 & 7 & 8 & 9 & 10 \\
\hline
 $\alpha_k$ & 0.836008  & 0.857487 & 0.876610 &0.891543 &  0.903259 &  0.912664 & 0.920367 & 0.926788  \\ \hline
\end{tabular}
\caption{ \label{table:guarantees} MAX-$k$-CUT approximation guarantees for $3\le k \le
10$}
\end{center}
}
\end{table}

\item
Limits of the type \eqref{eq:forodd} are of independent interest if one replaces the
$k$-page crossing number by the rectilinear crossing number. (The rectilinear crossing number of a graph
is the minimum number of edge crossings in  a drawing of the graph in the plane if all edges are drawn by straight lines.)
Indeed, for the rectilinear crossing number $ \overline{\Cr} (K_n)$, the limit
\[
\lim_{n \rightarrow \infty} \frac{\overline{\Cr}(K_n)}{
  \binom{n}{4}}
\]
is related to the Sylvester four point problem in geometric probability  as follows.
Consider an open set $R \in \mathbb{R}^2$ with finite area.
Denote by $q(R)$ the probability that the convex hull of four points in $R$, drawn uniformly at random, is a convex quadrilateral (as opposed to a line or triangle).
Scheinerman and Wilf \cite{Scheinerman-Wilf} showed that
\[
\inf_R q(R) = \lim_{n \rightarrow \infty} \frac{\overline{\Cr}(K_n)}{  \binom{n}{4}},
\]
where the infimum is taken over all open sets $R$ in the plane with finite area.

It remains an interesting question whether these limits also have alternative interpretations if one replaces the rectilinear crossing number
by other notions of crossing numbers, like the $k$-page crossing number.

In the second part of this work we will investigate the
$k$-page crossing numbers of certain complete bipartite graphs. We
will once again use optimization techniques, but the details are
somewhat different from those presented here, and are therefore best
presented seperately.

%
%

\paragraph{Acknowledgements.}
The authors are grateful to  Imrich Vrt'o for helpful comments.

 \end{document}